\documentclass{amsart}
\usepackage{amssymb,amsmath,amsthm}
\usepackage{mathrsfs}
\usepackage[pagebackref]{hyperref}
\usepackage{enumerate}

\renewcommand*{\backref}[1]{}
\renewcommand*{\backrefalt}[4]{%
  \ifcase #1 (Not cited.)%
  \or        (Page~#2.)%
  \else      (Pages~#2.)%
  \fi}

\newtheorem{theorem}{Theorem}[section]
\newtheorem{prop}[theorem]{Proposition}
\newtheorem{lemma}[theorem]{Lemma}
\theoremstyle{remark}
\newtheorem{remark}[theorem]{Remark}

\theoremstyle{definition}
\newtheorem{definition}[theorem]{Definition}
\newtheorem{cor}[theorem]{Corollary}

\numberwithin{equation}{section}

\def\Xint#1{\mathchoice
  {\XXint\displaystyle\textstyle{#1}}%
  {\XXint\textstyle\scriptstyle{#1}}%
  {\XXint\scriptstyle\scriptscriptstyle{#1}}%
  {\XXint\scriptscriptstyle\scriptscriptstyle{#1}}%
  \!\int}
\def\XXint#1#2#3{{\setbox0=\hbox{$#1{#2#3}{\int}$}
    \vcenter{\hbox{$#2#3$}}\kern-.5\wd0}}

\def\dashint{\Xint-}

\def\cD{\mathcal D}

\def\cR{\mathcal R}

\def\C{\mathbb C}

\def\N{\mathbb N}
\def\R{\mathbb R}

\def\Z{\mathbb Z}

\def\vol{\operatorname{Vol}}

\def\ric{\mathrm{Ric}}
\def\dgh{d_{\mathrm{GH}}}

\author{Shih-Kai Chiu}
\address{Mathematical Institute, University of Oxford, UK}
\email{shih-kai.chiu@maths.ox.ac.uk}

\title[Subquadratic harmonic functions on Calabi-Yau
manifolds]{Subquadratic harmonic functions on Calabi-Yau manifolds
  with maximal volume growth}
\begin{document}
\maketitle
\begin{abstract}
  On a complete Calabi-Yau manifold $M$ with maximal volume growth, a
  harmonic function with subquadratic polynomial growth is the real
  part of a holomorphic function. This generalizes a result of
  Conlon-Hein. We prove this result by proving a Liouville type
  theorem for harmonic $1$-forms, which follows from a new local $L^2$
  estimate of the exterior derivative.
\end{abstract}


\section{Introduction}
In this paper, we define a Calabi-Yau manifold to be a Ricci-flat
K\"ahler manifold. Let $(M,\omega)$ be a complete noncompact
Calabi-Yau manifold of complex dimension $n$. We say $(M,\omega)$ has
maximal volume growth if there exists $v>0$ such that for $p \in M$
and $r>0$, we have
\begin{equation*}
  \vol(B(p, r)) \ge vr^{2n}.
\end{equation*}

The study of complete noncompact Calabi-Yau manifolds dates back to
the foundational papers of Tian-Yau \cite{TianYauI, TianYauII}, in
which they construct complete Ricci-flat K\"ahler metrics on the
complement of a neat, almost ample divisor in a projective variety. An
important class of examples is asymptotically conical (AC) Calabi-Yau
manifolds. An AC Calabi-Yau manifold is a complete noncompact
Ricci-flat K\"ahler manifold such that outside a compact subset, the
manifold is diffeomorphic to a Ricci-flat K\"ahler cone, and the
metric on the manifold is (polynomially) asymptotic to the cone
metric.  Explicit important examples of AC Calabi-Yau manifolds
include smoothing and small resolutions of the ordinary double point;
see e.g. \cite{Candelas}\cite{Stenzel}. For a Calabi-Yau manifold $M$
with maximal volume growth, Cheeger-Tian \cite{CheegerTian} show that
if one tangent cone at infinity $M$ satisfies an integrability
condition, which implies smoothness, then $M$ is AC Calabi-Yau. We
refer the reader to Conlon-Hein \cite{ConlonHeinI, ConlonHeinII,
  ConlonHeinIII} for important results of AC Calabi-Yau manifolds.

The AC condition is restrictive, as it implies the tangent cone at
infinity is both smooth and unique. There are complete noncompact
Calabi-Yau manifolds with maximal volume growth which are not AC
Calabi-Yau. For example, Joyce's QALE manifolds \cite{Joyce} have
$\C^n/\Gamma$ as tangent cone at infinity, where
$\Gamma \subset \mathrm{SU}(n)$ is a discrete subgroup which does not
act freely on $\C^n$. Recently, Li \cite{Li}, Conlon-Rochon
\cite{ConlonRochon} and Sz\'ekelyhidi \cite{Szekelyhidi} independently
construct non-flat Calabi-Yau metrics with maximal volume growth on
$\C^3$; the last two groups have various generalizations of the
construction to higher dimensions. These metrics are not AC, since the
tangent cone at infinity of these metrics is $\C \times A_1$, where
$A_1$ denotes the singularity
\begin{equation*}
  \C^2 / \Z_2 \simeq \{ z_1^2 + z_2^2 + z_3^2 = 0 \} \subset \C^3,
\end{equation*}
equipped with the flat cone metric. This motivates us to try to
understand the relationship between these metrics. More generally, we
would like to study deformations of Calabi-Yau metrics on a Calabi-Yau
manifold with maximal volume growth. Since such manifolds also serve
as local models in various gluing constructions of compact Calabi-Yau
manifolds, the question of uniqueness is also natural in this respect.

In Conlon-Hein \cite{ConlonHeinI}, the rigidity of AC Calabi-Yau
metrics is shown by first showing that a harmonic function with
subquadratic growth is pluriharmonic. This particular result can be
seen as the linearized version of the rigidity of the complex
Monge-Amp\`ere equation. In this paper, we generalize this result (see
Theorem 3.8 and Corollary 3.9 in \cite{ConlonHeinI}) to the following:

\begin{theorem}
  \label{maintheorem}
  Let $M$ be a complete noncompact Ricci-flat manifold with maximal
  volume growth. Let $u$ be a harmonic $1$-form on $M$, and by this we
  mean that $(dd^*+d^*d)u=0$. Suppose $u$ has sublinear growth,
  i.e. there exist constants $C>0$ and $s<1$ such that
  \begin{equation*}
    |u| \le C(1+r)^s,
  \end{equation*}
  where $r$ is the distance function from a fixed point $p \in
  M$. Then $u = df$, where $f$ is a subquadratic harmonic function.
\end{theorem}

Theorem \ref{maintheorem} can be seen as an analogue of the Liouville
theorem of Cheng \cite{Cheng} that every harmonic function with
sublinear growth on a complete manifold with nonnegative Ricci
curvature is constant. As a corollary, we have

\begin{theorem}
  \label{pluriharmonictheorem}
  Let $M$ be a complete noncompact Calabi-Yau manifold with maximal
  volume growth. Then any subquadratic harmonic function on $M$ is the
  real part of a subquadratic holomorphic function. In particular, it is
  pluriharmonic.
\end{theorem}

\begin{remark}
  \label{firstremark}
  (1) When $M$ is AC Calabi-Yau, the proof in \cite{ConlonHeinI} shows
  that actually we can replace the subquadratic growth condition with
  the weaker condition $o(r^2)$ as $r \to \infty$. See Remark
  \ref{endremark} for details. It is interesting to know whether
  Theorem \ref{pluriharmonictheorem} still holds when we only assume
  $o(r^2)$.  (2) Theorem \ref{maintheorem} and
  \ref{pluriharmonictheorem} are false without the assumption of
  maximal volume growth. One counterexample is the Taub-NUT manifold,
  which is $\C^2$ equipped with a hyperk\"ahler metric with cubic
  volume growth. There is a harmonic function with linear growth which
  is not pluriharmonic. See Remark 3.10 in \cite{ConlonHeinI} for more
  details.
\end{remark}

\begin{proof}[Proof of Theorem \ref{pluriharmonictheorem}]
  Let $f$ be a subquadratic harmonic function. Then
  $d^cf = i(\bar{\partial}-\partial)f$ is a sublinear harmonic $1$-form
  by the Cheng-Yau gradient estimate~\cite{ChengYau}. So $d^cf = dg$ by
  Theorem \ref{maintheorem}. It follows that $f$ is the real part of the
  subquadratic holomorphic function $f+ig$.
\end{proof}

In \cite{ConlonHeinI}, the main idea is to compare the AC manifold
with its tangent cone at infinity. A lemma of Cheeger-Tian
\cite{CheegerTian}, roughly speaking, says that Theorem
\ref{maintheorem} is true on any cones with nonnegative Ricci
curvature. The comparison is then carried out using weighted function
spaces and the polynomial convergence of the metrics. In general we do
not have such linear theory, and the tangent cone at infinity might
have non-isolated singularities.

One approach to tackle the general case of maximal volume growth would
be to show that the spaces of harmonic functions with polynomial
growth on the manifold have the same dimensions as the corresponding
spaces on any tangent cone at infinity, and that the exterior
derivative mapping subquadratic harmonic functions to sublinear
harmonic $1$-forms is surjective. In general this is only possible if
we assume unique tangent cone at infinity. However in the K\"ahler
case, tangent cones at infinity are affine algebraic
\cite{LiuSzekelyhidi}. As a consequence, part of the harmonic spectrum
is discrete, and this is sufficient for establishing
Theorem~\ref{maintheorem} in the K\"ahler case. We refer to
\cite[Section~4.2]{ChiuThesis} for a proof following this line of
thought.

In this paper we follow another approach, which is to show that any
sublinear polynomial growth harmonic $1$-form is both closed and
coclosed. Once this is shown, then a result of
Anderson~\cite{Anderson90} showing that $H^1(M) = 0$ implies that such
harmonic $1$-forms must be exact. As an analogue of the proof of the
Liouville theorem for harmonic functions~\cite{Cheng}, we prove this
Liouville type theorem for harmonic $1$-forms by establishing the
following $L^2$ analogue of the Cheng-Yau gradient
estimate~\cite{ChengYau}:

\begin{theorem}
  \label{local_l2}
  Let $B(p,2)$ be a metric ball with $\vol(B(p,1)) \ge v > 0$ and
  $|\mathrm{Ric}| \le 1$. Then for any $\delta > 0$, there exists a
  constant $C > 0$ depending on $\delta$ and $v$ such that for any
  harmonic $1$-form $u$ on $B(p,1)$, we have
  \begin{equation*}
    \dashint_{B(p,r)} |du|^2 + |d^*u|^2 \le C r^{-2\delta} \dashint_{B(p,1)} |u|^2
  \end{equation*}
  for any $r \in (0,1/2]$.
\end{theorem}

\begin{remark}
  (1) Actually, we have a pointwise bound for $|d^*u|^2$ in terms of
  the $L^2$ average of $u$. See Remark \ref{dstar}. (2) The above
  estimate misses the traceless symmetric part of the gradient
  $\nabla u$, but is enough for our purpose.
\end{remark}

By a rescaling argument, we obtain Theorem~\ref{maintheorem} from
Theorem~\ref{local_l2}. Another immediate corollary of Theorem
\ref{local_l2}, which might be of independent interest, is the
following $L^2$ estimate of the complex Hessian of a harmonic
function:

\begin{theorem}
  \label{ddbar}
  Let $(X,\omega)$ be a K\"ahler manifold. Let $B(p,2)$ be a metric ball
  in $X$ with $\vol(B(p,1)) > v > 0$ and $|\ric| \le 1$. Then for any
  $\delta > 0$, there exists a constant $C > 0$, depending on $\delta$
  and $v$, such that for any harmonic function $f$ on $B(p,1)$, we have
  \begin{align*}
    \dashint_{B(p,r)} |\partial\bar{\partial} f|^2 \le C r^{-2\delta}
    \dashint_{B(p,1)} f^2
  \end{align*}
  for any $r \in (0,1/2]$.
\end{theorem}

\begin{remark}
  The counterexample of the Taub-NUT manifold, as discussed in Remark
  \ref{firstremark}, shows that the dependence on the volume lower bound
  $v$ in both Theorem \ref{local_l2} and Theorem \ref{ddbar} cannot be
  removed. This is because we can apply a rescaling argument to obtain
  Theorem \ref{maintheorem} from either Theorem \ref{local_l2} or
  Theorem \ref{ddbar}.
\end{remark}

The proof of Theorem \ref{local_l2}, roughly speaking, reduces to a
statement about harmonic $1$-forms on cones by Cheeger-Colding theory
\cite{CheegerColding}. Thus the Cheeger-Tian lemma mentioned above
plays a crucial role. Another technical ingredient is the monotonicity
of frequency functions of harmonic forms on Ricci-flat cones
(Proposition \ref{monotonicity_singular}), which implies $L^2$ three
circle theorems (Theorem \ref{threecircle}) for harmonic functions and
$1$-forms. The frequency function is a direct generalization of the
one defined in Colding-Minicozzi \cite{ColdingMinicozziI} on cones for
the function case, which in turn is a generalization of the original
version on $\R^n$ defined by Almgren \cite{Almgren}. Frequency
functions have been studied and applied intensively \cite{Almgren1,
  GarofaloLin, GromovSchoen}, and so are various three circle theorems
\cite{ding2004existence, Xu, DonaldsonSunII, Liu}. The relationship
between the monotonicity of frequency functions and the three circle
theorem is made precise and used in Lin \cite{Lin}.

This paper is organized as follows. In Section~2, we prove preliminary
results about harmonic forms on limit Ricci-flat cones. In Section~3,
we prove the monotonicity of Almgren's frequency functions of harmonic
$k$-forms on a limit Ricci-flat cone. In Section~4, we prove the
aforementioned lemma of Cheeger-Tian \cite{CheegerTian} in our
singular setting. Finally in Section~5, we prove Theorem
\ref{local_l2}, and deduce Theorem \ref{maintheorem} and Theorem
\ref{pluriharmonictheorem} as corollaries.

We end this introduction by reviewing the notion of tangent cones at
infinity. Let $(M,g)$ be a complete noncompact Riemannian manifold $M$
with nonnegative Ricci curvature and maximal volume growth, and let
$p \in M$ be a fixed point. Given a sequence of positive number
$r_i \to \infty$, we consider the sequence of pointed Riemannian
manifolds $(M_i, p_i, g_i) = (M, p, r_i^{-2}g)$. By Gromov compactness
theorem \cite{Gromov}, after passing to a subsequence, the sequence
$(M_i, p_i, g_i)$ converges to a complete metric space $(T_\infty, d)$
in the pointed Gromov-Hausdorff sense. Cheeger-Colding theory
\cite{CheegerColdingI} tells us that the limit space $(T_\infty, d)$
is actually a metric cone. We call $(T_\infty, d)$ a tangent cone at
infinity of $M$. If the metric $g$ is Ricci-flat, then the results of
Anderson \cite{Anderson} and Cheeger-Naber \cite{CheegerNaber}
combined show the following:
\begin{itemize}
\item the regular set is open and dense,
\item the singular set of $T_\infty$ is closed of Minkowski
  codimension at least $4$, and
\item on the regular set, the metrics $g_i$ converge in $C^\infty$ to
  a Ricci-flat metric $g_\infty$.
\end{itemize}

The above regularity properties, except the Minkowski content
estimate, were also proved in the K\"ahler setting by
Cheeger-Colding-Tian \cite{CheegerColdingTian}. Note that when we only
assume nonnegative Ricci curvature, the tangent cones at infinity are
in general not unique \cite{Perelman} . However, if the metric is
Ricci-flat and if one tangent cone at infinity is smooth, then it is
unique by Colding-Minicozzi \cite{ColdingMinicozziIII}.

\emph{A remark on notations.}  In this paper we always use the Hodge
Laplacian $\Delta = dd^*+d^*d$ for arbitrary $k$-forms, including
$k=0$, so in particular the eigenvalues are nonnegative. A harmonic
$k$-form is a $k$-form $u$ such that $\Delta u = 0$. What we mean by
subquadratic is in the sense of $O(r^s)$ for some $s<2$. The same goes
for the term sublinear.

\hfill

\noindent\textbf{Acknowledgments.} This work grew out of my PhD thesis.
I would like to thank G\'abor Sz\'ekelyhidi for kindly sharing his
valuable insights and many helpful discussions over the years. I would
also like to thank the anonymous referees for very helpful
comments. Finally, I would like to thank the National Center for
Theoretical Sciences in Taipei, Taiwan for their warm hospitality
during summer 2019.

\section{Analysis on limit Ricci-flat cones}

Before moving on, we shall remark that all the results in this section
also hold for smooth Riemannian cones.

\begin{definition}
  We define a limit Ricci-flat cone $C(Y)$ to be a metric cone over a
  compact metric space $Y$, such that $C(Y)$ is itself a pointed
  Gromov-Hausdorff limit of a noncollapsing sequence of complete
  Riemannian manifolds $(M_i,p_i,g_i)$ with
  $|\mathrm{Ric}(g_i)| \le \epsilon_i \to 0$.
\end{definition}

In particular, a tangent cone at infinity of a Ricci-flat manifold
with maximal volume growth is a limit Ricci-flat cone. The
regularity results mentioned in the introduction enable us to extend
various results in the smooth case to our singular setting. We remark
that the condition $\ric(g_i) \to 0$ in the definition, though weaker
than the condition mentioned in the introduction, is strong enough to
ensure that the limit metric is smooth and Ricci-flat on the regular
set. This follows from elliptic regularity of the Ricci curvature
equation in harmonic coordinates. See for example the argument in the
proof of \cite[Main~Lemma~2.2]{AndersonConvergence}.

A key consequence of these regularity results is the existence of good
cutoff functions. First, let us note that a limit Ricci-flat cone
$C(Y)$ is an $\operatorname{RCD}^*(0,m)$-space and the cross section
$Y$ is an $\operatorname{RCD}^*(m-2,m-1)$-space \cite{Ketterer}, where
$m$ is the (real) dimension of $C(Y)$. $\operatorname{RCD}$ spaces are
generalizations of limits of Riemannian manifolds with Ricci curvature
bounded from below. See \cite{GigliLecture} for an introduction. We
have the following lemma.

\begin{lemma}[Mondino-Naber \cite{MondinoNaber}]
  \label{cheegercoldingcutoff}
  Let $X$ be a $\operatorname{RCD}^*(K,N)$-space for some $K \in \R$ and
  $N \in (1,\infty)$. Then for every $x \in X, R>0, 0<r<R$, there exists
  a Lipschitz function $\psi^r: X \to \R$ satisfying:
  \begin{itemize}
  \item $\operatorname{supp} \psi^r \subset B(x,r)$,
  \item $\psi^r = 1 $ on $B(p, r/2)$,
  \item $r^2|\Delta \psi^r|+r|\nabla \psi^r| \le C(K,N,R)$.
  \end{itemize}
\end{lemma}

Using the existence of good cutoff functions, we can construct cutoff
functions that allow us to do analysis on spaces with codimension $4$
singularities.

\begin{lemma}
  \label{goodcutoff}
  Let $C(Y)$ be a limit Ricci-flat cone. Denote $\Sigma$ the singular
  set of $C(Y)$. Fix $p \in C(Y)$. Then for any $\epsilon >0$, there
  exists a cutoff function $\phi_\epsilon$ on $B(p,1)$ such that:
  \begin{itemize}
  \item $\operatorname{supp} \phi_\epsilon \subset \Sigma^{\epsilon}$,
    where $\Sigma^\epsilon$ is the $\epsilon$-neighborhood of $\Sigma$,
  \item $\phi_\epsilon = 1$ in a neighborhood of $\Sigma$,
  \item $\|\nabla\phi_\epsilon\|_{L^2} \to 0$ as $\epsilon \to 0$,
  \item $\|\Delta\phi_\epsilon\|_{L^1} \to 0$ as $\epsilon \to 0$,
  \item $\|\Delta\phi_\epsilon\|_{L^2} < C$ for some constant $C>0$
    not depending on $\epsilon$.
  \end{itemize}
  The exact same construction also holds on the cross section $Y$.
\end{lemma}

\begin{proof}
  The proof follows from a standard covering argument. See for example
  the proof of Proposition 3.5 in \cite{DonaldsonSunI}.
\end{proof}

We briefly recall some basic facts about the geometry of cones. Let
$C(Y)$ be a limit Ricci-flat cone of dimension $m$. Thanks to the
regularity results and the existence of good cutoff functions, we can
do analysis on the regular set $\mathcal{R}$ of $C(Y)$ or the regular
set of $Y$ as in the smooth case. In the rest of this paper, unless
otherwise stated, tensors, as well as operators acting on them, are
defined on the regular set $\mathcal{R}$. The Hodge Laplacian
$\Delta: \Omega^k(\cR) \to \Omega^k(\cR)$ is defined as
\begin{equation*}
  \Delta = d^*d+dd^*.
\end{equation*}

Recall the Bochner formula. Let $\omega$ be a differential
$k$-form. Then we have
\begin{equation*}
  \Delta \omega = \nabla^*\nabla \omega + \mathfrak{R}(\omega),
\end{equation*}
where $\mathfrak{R}$ is a $0$th-order self-adjoint differential
operator defined using the Riemann curvature tensor. If $k=1$, then
$\mathfrak{R}=\operatorname{Ric}$ is just the Ricci tensor.

Let $i,j,k$, etc. denote the indices of local coordinates on the
regular part of the cross section $Y$. The coordinate vector fields
$\partial_i$ on $Y$ extend trivially to vector fields on $\cR$. The
cone metric on $\cR$ is given by $g_{C(Y)} = dr^2 + r^2 g_Y$. Here $r$
is the radial coordinate. The Christoffel symbols of $g_{C(Y)}$ with
respect to these coordinates can be calculated as
\begin{align*}
  \Gamma_{ij}^k &= (\Gamma^Y)_{ij}^k, \\
  \Gamma_{ir}^k &= \frac{1}{r}\delta_i^k, \\
  \Gamma^r_{ij} &= -rg^Y_{ij},\\
  \Gamma_{ir}^r&=\Gamma_{rr}^i=\Gamma_{rr}^r= 0.
\end{align*}

Using these, the Riemann curvature tensor can be calculated as
follows:
\begin{align*}
  R(\partial_i, \partial_r) &= 0, \\
  R(\partial_i, \partial_j)\partial_k &= R^Y(\partial_i,\partial_j)\partial_k-g^Y_{jk}\partial_i+g^Y_{ik}\partial_j, \\
  R(\partial_i, \partial_j)\partial_r &= 0.
\end{align*}

We note that the Riemann curvature tensor is homogeneous of degree
$-2$:
\begin{align*}
  \nabla_{r\partial_r} R = -2 R.
\end{align*}

This follows from the fact that
\begin{align*}
  \nabla_{r\partial_r} \partial_i &= \partial_i, \\
  \nabla_{r\partial_r} dx^i &= -dx^i.
\end{align*}

We need the following definition:

\begin{definition}
  Let $C(Y)$ be a limit Ricci-flat cone, and let $\cR$ denote the
  regular set of $C(Y)$. Let $u$ be a harmonic $k$-form on $\cR$. We say
  $u$ is locally $L^p$ (resp. locally $W^{1,2}$) and write
  $u \in L^p_{loc}$ (resp. $u \in W^{1,2}_{loc}$) if for any
  $p \in C(Y)$, regular or not, and for any $r > 0$, we have
  \begin{align*}
    \int_{B(p,r)\cap \cR} |u|^p < \infty
  \end{align*}
  \begin{align*}
    (\text{resp. }\int_{B(p,r)\cap \cR} |u|^2 + \int_{B(p,r)\cap \cR}
    |\nabla u|^2 &< \infty).
  \end{align*}
  For $k$-forms on a smooth Riemannian cone (not necessarily
  Ricci-flat), the same definition follows.
\end{definition}

If the curvature operator $\mathfrak{R}$ on $C(Y)$ is nonnegative,
then by the Bochner formula $|u|$ is subharmonic for any harmonic
$k$-form $u$. We can then deduce locally $L^\infty$ from locally $L^1$
or locally $L^2$ by the following mean value inequality on limit
Ricci-flat cones.

\begin{prop}
  \label{linfinity}
  Let $C(Y)$ be a limit Ricci-flat cone of dimension $m$. Then there
  exists a constant $C>0$ depending on $m$ such that the following
  holds. Let $f \in L^2_{loc} \cap C^\infty(\cR)$ be a subharmonic
  function, i.e. $-\Delta f \ge 0$ on $\cR$. Then for any $x$ in $\cR$,
  we have
  \begin{align*}
    f(x) \le C \dashint_{B(x,1)} |f| \le C\sqrt{\dashint_{B(x,1)} f^2}.
  \end{align*}
\end{prop}

\begin{proof}
  For a proof on smooth manifolds, see
  \cite[Theorem~1.2]{LiSchoen}. Fix $x \in \cR$ a regular point. Let
  $H(x,y,t)$ be the heat kernel on $C(Y)$. Since $H$ is the uniform
  limit of the heat kernels on manifolds along the
  sequence~\cite{ding2002heat}, $H$ satisfies the Gaussian bounds and
  time derivative bounds in \cite[Theorem~5.4.12]{SaloffCoste}. Using
  the Li-Yau inequality~\cite{LiYau}, we also obtain a Gaussian upper
  bound for $|\nabla H|$. In sum, we have bounds:
  \begin{equation*}
    |H(x,y,t)| + \sqrt{t}|\nabla_yH(x,y,t)| \le
    \frac{C'}{\vol(B(x,\sqrt{t}))}\exp\left(-\frac{c' d(x,y)^2}{t}\right)
  \end{equation*}
  for $t \in (0,1]$. Here $C',c'>0$ are dimensional constants. Let
  $\eta$ be a cutoff function on $B(x,1)$ given in Lemma
  \ref{cheegercoldingcutoff} with $r=1/2$. Recall that
  $\|\nabla\eta\|_{L^\infty}, \|\Delta\eta\|_{L^\infty} < C$ where
  $C>0$ depends only on the dimension $m$. Now we compute
  \begin{equation}
    \label{eq:27}
    \begin{aligned}
      \frac{\partial}{\partial t}\int_{B(x,1)} Hf\eta
      &= \int_{B(x,1)} (-\Delta H)f\eta \\
      &= \int_{B(x,1)} -f\Delta(H\eta) + fH\Delta\eta - 2f\nabla\eta \cdot \nabla H \\
      &\ge -\int_{B(x,1)} f\Delta(H\eta)-C''\dashint_{B(x,1)}|f|,
    \end{aligned}
  \end{equation}
  where $C'' > 0$ is a dimensional constant. Here in the last inequality
  we have used the uniform bounds for $\nabla\eta, \Delta\eta,H$ and
  $\nabla H$ on the annulus $A(x,1/2,1)$, which contains the support of
  $\nabla\eta$ and $\Delta\eta$, and volume comparison. It remains to
  conclude that
  \begin{equation}
    \label{eq:28}
    -\int_{B(x,1)} f\Delta(H\eta) \ge 0
  \end{equation}
  for all $t>0$. Let $\phi_\epsilon$ be the cutoff function given in
  Lemma~\ref{goodcutoff}, and set $\psi_\epsilon =
  (1-\phi_\epsilon)$. Recall that $\|\nabla\psi\|_{L^2} \to 0$ and
  $\|\Delta\psi\|_{L^2} \le C$ for a constant $C$ independent of
  $\epsilon$. We compute
  \begin{align*}
    -\int_{B(x,1)} \psi_\epsilon f\Delta(H\eta)
    &= -\int_{B(x,1)} f\Delta(\psi_\epsilon H\eta)
    - \int_{B(x,1)} fH\eta\Delta\psi_\epsilon
    - 2\int_{B(x,1)} f\nabla(H\eta)\cdot\nabla\psi_\epsilon \\
    &\ge
    - \int_{B(x,1)} fH\eta\Delta\psi_\epsilon
    - 2\int_{B(x,1)} f\nabla(H\eta)\cdot\nabla\psi_\epsilon
  \end{align*}
  Here in the second inequality we perform an integration by parts and
  use the fact that $-\Delta f \ge 0$. Using the Cauchy-Schwarz
  inequality and the properties of $\psi_\epsilon$, we conclude
  \eqref{eq:28} by letting $\epsilon \to 0$. Finally, integrating
  \eqref{eq:27} on $[0,1]$ and using the Gaussian upper bound again,
  we get
  \begin{align*}
    C'\dashint_{B(x,1)}f \ge \int_{B(x,1)} H(x,y,1)f\eta \ge f(x) -C''\dashint_{B(x,1)}|f|.
  \end{align*}
  This concludes the proof.
\end{proof}

\begin{prop}
  \label{w12}
  On a limit Ricci-flat cone $C(Y)$ with nonnegative curvature operator
  $\mathfrak{R}$, a harmonic $k$-form which is locally $L^\infty$ is locally
  $W^{1,2}$.
\end{prop}

\begin{proof}
  Let $\eta$ be a cutoff function given in Lemma
  \ref{cheegercoldingcutoff} on $B_2$ such that $\eta = 1$ on $B_1$. Let
  $\phi_\epsilon$ be a cutoff function given in Lemma
  \ref{goodcutoff}. For a locally $L^\infty$ harmonic $k$-form $u$,
  \begin{align*}
    \int_{B_1} (1-\phi_\epsilon) |\nabla u|^2
    &\le \int_{B_2} \eta (1-\phi_\epsilon) |\nabla u|^2 \\
    &\le \int_{B_2} -\frac{1}{2}\eta (1-\phi_\epsilon) \Delta |u|^2 \\
    &= \int_{B_2} -\frac{1}{2}[(1-\phi_\epsilon)\Delta \eta + 2\nabla \eta \cdot
    \nabla\phi_{\epsilon} - \eta \Delta\phi_{\epsilon}] |u|^2 \\
    &\le C\int_{B_2} (1-\phi_\epsilon) |u|^2
    + C\int_{B_2} |\nabla \phi_\epsilon| |u|^2
    + C\int_{B_2} |\Delta\phi_\epsilon| |u|^2.
  \end{align*}
  By the Cauchy-Schwarz inequality and locally $L^\infty$ assumption on
  $u$, the last two terms on the right hand side vanish as
  $\epsilon \to 0$.
\end{proof}

The following is needed in order to bound $|\nabla_{r\partial_r} u|$.

\begin{prop}
  \label{subharmonic}
  Let $C(Y)$ be a limit Ricci-flat cone with nonnegative curvature
  operator $\mathfrak{R}$. Then $|\nabla_{r\partial_r}u|$ is subharmonic
  for any harmonic $k$-form $u$ on $C(Y)$.
\end{prop}

\begin{proof}
  Using the Bochner formula and the cone structure of $C(Y)$, by a
  straightforward calculation, we get
  \begin{align*}
    -\Delta |\nabla_{r\partial_r} u|^2 =
    2\mathfrak{R}(\nabla_{r\partial_r} u, \nabla_{r\partial_r} u) +
    2|\nabla \nabla_{r\partial_r} u|^2 \ge 2|\nabla |\nabla_{r\partial_r} u||^2.
  \end{align*}
  Now set $h = |\nabla_{r\partial_r}u|$. From the above we have
  \begin{align*}
    -\Delta h^2 = 2|\nabla h|^2  -2h\Delta h \ge 2|\nabla h|^2,
  \end{align*}
  and so $-h\Delta h \ge 0$. Let $p \in \cR$. If $h(p)=0$, then
  necessarily $-\Delta h (p) \ge 0$, for $p$ is a minimum. If
  $h(p) > 0$, then we also have $-\Delta h(p) \ge 0$. The result
  follows.
\end{proof}

\begin{lemma}
  \label{nablar}
  Let $C(Y)$ be a limit Ricci-flat cone and let $u$ be a locally
  $L^\infty$ harmonic $1$-form. Then locally
  $\sup |\nabla_{r\partial_r} u| < \infty$.
\end{lemma}

\begin{proof}
  This follows from Proposition~\ref{linfinity}, Proposition~\ref{w12}
  and Proposition~\ref{subharmonic}.
\end{proof}

\section{Monotonicity of the frequency function}

Fix for now a (smooth) Riemannian cone $C(Y)$ of real dimension
$m \ge 2$. Later on, we will generalize what we obtained for smooth
cones to limit Ricci-flat cones using cutoff functions. In this
section, $B_r$ will denote the open ball of radius $r$ centered at the
vertex of $C(Y)$, and $\partial B_r = \{r\} \times Y$ will denote the
boundary of $B_r$. Let $u$ be a harmonic $k$-form such that locally
$\sup |u| + \sup |\nabla_{r\partial_r}u| < \infty$. Define
\begin{equation*}
  D(r) = \int_{B_r} |\nabla u|^2 + \mathfrak{R}(u,u).
\end{equation*}
and
\begin{equation*}
  H(r) = \int_{\partial B_r} |u|^2.
\end{equation*}
In analogy with the frequency of harmonic functions (\cite{Han}), we
define
\begin{equation*}
  N(r) = \frac{rD(r)}{H(r)}
\end{equation*}
whenever $H(r) > 0$. $N(r)$ is called the frequency function of $u$.

By the Bochner formula,
\begin{equation*}
  \frac{1}{2}\nabla \cdot \nabla |u|^2 = |\nabla u|^2 + \mathfrak{R}(u,u).
\end{equation*}
So
\begin{equation}
  \label{dalternative}
  D(r) = \frac{1}{2} \int_{B_r} \nabla \cdot \nabla |u|^2 = \int_{\partial B_r} u \cdot \nabla_{\partial_r} u,
\end{equation}
Note that unlike in the $\R^m$ case, the vertex of the cone is a
singular point. To justify the integration by parts in
\eqref{dalternative}, we can cut out a small ball centered at the
vertex of the cone and then letting the radius of the small ball go to
$0$. Hence we need the assumption that locally
$\sup |u| + \sup |\nabla_{r\partial_r} u| < \infty$.

We now prove the monotonicity of frequency functions.

\begin{prop}
  \label{monotonicity}
  Let $C(Y)$ be a smooth Riemannian cone of dimension $m \ge 2$, let $u$
  be a harmonic $k$-form on $C(Y)$ such that locally
  $\sup |u| + \sup |\nabla_{r\partial_r}u| < \infty$, and let $N(r)$ be
  the frequency function of $u$. Then $N'(r) \ge 0$. $N(r)$ is constant
  if and only if
  \begin{equation*}
    \nabla_{r\partial_r} u = h(r)u
  \end{equation*}
  for some function $h(r)$.
\end{prop}

The following proof is a slight modification of the case of harmonic
functions on $\R^m$. See \cite{Han} for comparison.

\begin{proof}
  By direct differentiation,
  \begin{equation*}
    N'(r) = N(r) \left\{\frac{1}{r} + \frac{D'(r)}{D(r)} - \frac{H'(r)}{H(r)}\right\}.
  \end{equation*}
  So our goal is to show that
  \begin{equation*}
    \frac{1}{r} + \frac{D'(r)}{D(r)} - \frac{H'(r)}{H(r)} \ge 0.
  \end{equation*}
  First we calculate $D'(r)$.
  \begin{align*}
    D'(r) &= \int_{\partial B_r} |\nabla u|^2+\mathfrak{R}(u,u)
    = \frac{1}{r} \int_{\partial B_r} \left(|\nabla u|^2+\mathfrak{R}(u,u)\right)r\partial_r \cdot \partial_r \\
    &= \frac{1}{r} \int_{B_r} \nabla \cdot \left(|u|^2+\mathfrak{R}(u,u)\right)r\partial_r.
  \end{align*}
  The divergence inside the integral can be calculated as follows. Fix
  normal coordinates on $Y$, and denote the indices of the normal
  coordinates by $i,j$. In particular, $\partial_r, r^{-1}\partial_i$
  form a local orthonormal frame on $C(Y)$. We have
  \begin{align*}
    \nabla \cdot \left(|\nabla u|^2+\mathfrak{R}(u,u)\right)r \frac{\partial}{\partial r}
    &= \frac{1}{\sqrt{\det{g}}}\partial_r\left(\sqrt{\det{g}} \;\left(|\nabla u|^2+\mathfrak{R}(u,u)\right) r\right) \\
    &= m\left(|\nabla u|^2+\mathfrak{R}(u,u)\right) + r\partial_r\left(|\nabla u|^2+\mathfrak{R}(u,u)\right).
  \end{align*}
  For the third term on the right hand side, we have
  \begin{align*}
    r\partial_r |\nabla u|^2 &= r\partial_r\left(|\nabla_{\partial_r}u|^2+r^{-2}|\nabla_i u|^2\right) \\
    &=2 r\nabla_{\partial_r}\nabla_{\partial_r}u \cdot \nabla_{\partial_r}u -2\left(|\nabla u|^2 - |\nabla_{\partial_r}u|^2\right) +
    2r^{-2}\nabla_{r\partial_r}\nabla_iu \cdot \nabla_i u \\
    &=2 \nabla_{\partial_r}\nabla_{r\partial_r}u \cdot \nabla_{\partial_r}u -2|\nabla u|^2 +
    2r^{-2}\nabla_{r\partial_r}\nabla_iu \cdot \nabla_i u.
  \end{align*}
  Using the fact that
  $\nabla_{\partial_r}\nabla_i = \nabla_i\nabla_{\partial_r}$ and
  combining the first and third term on the right hand side, we get
  \begin{align*}
    r\partial_r |\nabla u|^2 = 2 \nabla \nabla_{r\partial_r}u \cdot \nabla u -2|\nabla u|^2.
  \end{align*}

  By homogeneity of the curvature tensor, we also have
  \begin{align*}
    r\partial_r\mathfrak{R}(u,u) = -2\mathfrak{R}(u,u) + 2\mathfrak{R}(\nabla_{r\partial_r}u,u).
  \end{align*}

  Combining these, the divergence term gives
  \begin{align*}
    \nabla \cdot \left(|\nabla u|^2+\mathfrak{R}(u,u)\right)r \partial_r
    &= (m-2)\left(|\nabla u|^2 + \mathfrak{R}(u,u)\right) +
    2\left(\nabla \nabla_{r\partial_r}u \cdot \nabla u + \mathfrak{R}(\nabla_{r\partial_r}u,u)\right).
  \end{align*}

  Finally,
  \begin{align}
    D'(r) &= \frac{m-2}{r} D(r)
    + \frac{2}{r}\int_{B_r}\nabla \nabla_{r\partial_r}u \cdot \nabla u
    + \mathfrak{R}(\nabla_{r\partial_r}u,u) \label{dprime1}\\
    &= \frac{m-2}{r} D(r)
    + 2\int_{\partial B_r} |\nabla_{\partial_r}u|^2.\label{dprime}
  \end{align}
  The last equality follows from the Bochner formula and $u$ being
  harmonic.

  Next we work on $H(r)$.
  \begin{align*}
    H'(r) &= \frac{(m-1)}{r}H(r) + 2\int_{\partial B_r} u \cdot \nabla_{\partial_r}u.
  \end{align*}

  So
  \begin{align*}
    \frac{1}{r} + \frac{D'(r)}{D(r)} - \frac{H'(r)}{H(r)}
    &= 2\left(\frac{\int_{\partial B_r} |\nabla_{\partial_r}u|^2}{\int_{\partial B_r} u \cdot \nabla_{\partial_r} u} - \frac{\int_{\partial B_r} u \cdot \nabla_{\partial_r} u}{\int_{\partial B_r} |u|^2} \right) \ge 0
  \end{align*}
  by the Cauchy-Schwarz inequality. Equality holds if and only if
  \begin{align*}
    \nabla_{r\partial_r}u =h(r)u
  \end{align*}
  for some function $h(r)$.

\end{proof}

We now turn to the singular setting. Suppose now that $C(Y)$ is a
limit Ricci-flat cone. By defining
$D(r) = \int_{\partial B_r} u \cdot \nabla_{\partial_r} u$ as in
\eqref{dalternative}, we see that $D(r)$ is well defined, provided
that the harmonic $k$-form $u$ is such that locally
$\sup |u| + \sup |\nabla_{r\partial_r}u| < \infty$. The integration by
parts in \eqref{dalternative} is then justified using cutoff functions
in Lemma~\ref{goodcutoff}. The precise argument is similar to the
proof of the Proposition below, so we omit it. We now generalize
Proposition \ref{monotonicity} to the case of singular cross sections.

\begin{prop}
  \label{monotonicity_singular}
  Let $C(Y)$ be a limit Ricci-flat cone of dimension $m \ge 2$. Let $u$
  a harmonic $k$-form on $C(Y)$ such that locally
  $\sup{|u|} + \sup{|\nabla_{r\partial_r} u|} < \infty$. Then $N(r)$ is
  well-defined and differentiable. Moreover, $N'(r) \ge 0$. $N(r)$ is
  constant if and only if $\nabla_{r\partial_r} u = h(r) u$ for some
  function $h(r)$.
\end{prop}

\begin{proof}
  We need to show that both $H(r)$ and $D(r)$ are differentiable. We
  will focus on $D(r)$ as it is more involved. Let $\phi_\epsilon $ be
  the cutoff function on $Y$ given in Lemma~\ref{goodcutoff}. Define
  \begin{equation*}
    u_\epsilon = (1-\phi_\epsilon) u,
  \end{equation*}
  and define $D_\epsilon(r)$ to be $D(r)$ with $u$ replaced by
  $u_\epsilon$. From a straightforward calculation and using the
  properties of the cutoff functions $\phi_\epsilon$, we see that
  $\lim_{\epsilon \to 0} D_\epsilon(r) = D(r)$. The goal is to show that
  $D'(r)$ exists by showing that
  $D'(r) = \lim_{\epsilon\to 0} D'_\epsilon(r)$. We now calculate
  $D'_\epsilon(r)$. From \eqref{dprime1}, we have
  \begin{align*}
    D_\epsilon'(r) = \frac{m-2}{r} D_\epsilon(r) + \frac{2}{r}\int_{B_r}\nabla
    \nabla_{r\partial_r}u_\epsilon \cdot \nabla u_\epsilon +
    \mathfrak{R}(\nabla_{r\partial_r}u_\epsilon,u_\epsilon).
  \end{align*}
  We cannot apply \eqref{dprime} since $u_\epsilon$ is no longer
  harmonic. To estimate the second term on the right hand side, we
  integrate by parts and use the Bochner formula:
  \begin{equation*}
    \begin{aligned}
      \frac{1}{r}\int_{B_r} \nabla\nabla_{r\partial_r} u_\epsilon
      \cdot \nabla u_\epsilon +
      \mathfrak{R}(\nabla_{r\partial_r}u_\epsilon, u_\epsilon)
      &= \int_{\partial B_r} |\nabla_{\partial_r} u_\epsilon|^2 \\
      &\phantom{???}+ \frac{1}{r} \int_{B_r} \nabla_{r\partial_r} u_\epsilon \cdot \nabla^*\nabla u_\epsilon + \mathfrak{R}(\nabla_{r\partial_r}u_\epsilon, u_\epsilon) \\
      &= \int_{\partial B_r} (1-\phi_\epsilon)^2 |\nabla_{\partial_r}u|^2 \\
      &\hphantom{= }-\frac{1}{r}\int_{B_r} (1-\phi_\epsilon)\Delta\phi_\epsilon \nabla_{r\partial_r}u \cdot u \\
      &\hphantom{= }+\frac{2}{r} \int_{B_r} \nabla_{r\partial_r} u
      \cdot \left((1-\phi_\epsilon) \nabla \phi_\epsilon \cdot \nabla
        u\right).
    \end{aligned}
  \end{equation*}

  Here the integration by parts in the first equality is justified
  using the fact that locally
  $\sup |u| + \sup |\nabla_{r\partial_r}u| < \infty$ as in
  \eqref{dalternative}. To conclude that
  $\lim_{\epsilon \to 0} D'_\epsilon(r)$ exists and satisfies the
  right equation, it is enough to show that last two terms on the
  right hand side vanish as $\epsilon \to 0$. By the Cauchy-Schwarz
  inequality and the assumptions that locally
  $\sup{|u|}+\sup{|\nabla_{r \partial_r}u|} < \infty$, the problem
  reduces to showing that
  \begin{equation*}
    \lim_{\epsilon \to 0} \int_Y |\nabla \phi_\epsilon|^2 = 0, \:\:\:\:\:\:
    \lim_{\epsilon \to 0} \int_Y |\Delta\phi_\epsilon| = 0.
  \end{equation*}
  But these are the properties of $\phi_\epsilon$. In sum, we have shown that
  \begin{equation*}
    \lim_{\epsilon \to 0} D'_\epsilon(r) = \frac{m-2}{r}D(r) + \int_{\partial B_r} |\nabla_{\partial_r}u|^2
    = D'(r).
  \end{equation*}
  The rest follows from the proof of Proposition~\ref{monotonicity}.
\end{proof}


The monotonicity of the frequency implies the following $L^2$ three circle
theorem:

\begin{theorem}
  Let $C(Y)$ be a limit Ricci-flat cone of dimension $m \ge 2$. Let $u$
  be a harmonic $1$-form on $C(Y)$ such that locally
  $\sup |u| + \sup|\nabla_{r\partial_r}u| < \infty$. Then $H(r)$ is
  $\log$-convex with respect to $\log r$. Equality holds if and only if
  $\nabla_{r\partial_r} u = h(r) u$ for some function $h(r)$.
\end{theorem}

\begin{proof}
  \begin{align*}
    \frac{d \log H(r)}{d\log r}
    = \frac{r H'(r)}{H(r)}
    = (n-1) + 2 N(r).
  \end{align*}
  So
  \begin{align*}
    \frac{d^2 \log H(r)}{d (\log r)^2}
    = 2\frac{d N(r)}{d \log r}
    = 2rN'(r) \ge 0.
  \end{align*}
\end{proof}

Let
\begin{equation*}
  F(r) = \int_{B_r} u^2.
\end{equation*}
Integrating $H(r)$, we see that $F(r)$ also satisfies the three circle
theorem:

\begin{theorem}
  \label{threecircle}
  Let $C(Y)$ be a limit Ricci-flat cone of dimension $m \ge 2$. Let $u$
  be a harmonic $1$-form on $C(Y)$ such that locally
  $\sup |u| + \sup|\nabla_{r\partial_r}u| < \infty$. Then
  \[
    F(r) = \int_{B_r} u^2
  \]
  is $\log$-convex with respect to $\log r$. Equality holds if and only
  if $\nabla_{r\partial_r} u = h(r) u$ for some function $h(r)$.
\end{theorem}

\begin{proof}
  Let $0<r_1<r<r_2$ and let
  \begin{align*}
    \frac{1}{p} = \frac{\log{r_2}-\log{r}}{\log{r_2}-\log{r_1}},  \:\:\:\:\:\:
    \frac{1}{q} = \frac{\log{r}-\log{r_1}}{\log{r_2}-\log{r_1}}.
  \end{align*}
  Then we compute
  \begin{align*}
    F(r) &= \int_0^r H(s)ds \\
    &\le  \int_0^r H(r_1s/r)^{1/p} H(r_2s/r)^{1/q}ds \\
    &\le  \left(\int_0^r H(r_1s/r)ds\right)^{1/p} \left(\int_0^r H(r_2s/r)ds\right)^{1/q} \\
    &=
    F(r_1)^{1/p}F(r_2)^{1/q} \left(\frac{r}{r_1}\right)^{1/p} \left(\frac{r}{r_2}\right)^{1/q} \\
    &= F(r_1)^{1/p}F(r_2)^{1/q}.
  \end{align*}
\end{proof}

\section{Homogeneous $1$-forms}

Let us now focus on the case when the frequency function is
constant. In the case of harmonic $0$-forms, that is, when $u$ is a
harmonic function, the homogeneous condition
$\nabla_{r\partial_r} u = h(r)u$ is equivalent to
\begin{equation*}
  u = r^{s}g,
\end{equation*}
where $g$ is an eigenfunction on the cross section $Y$ with eigenvalue
$s(s+m-2)$. Thus the set of possible degrees $s$, denoted as
$\cD(C(Y))$, is determined by the spectrum of the cross section
$Y$. The following lemma, based on a lemma of Cheeger-Tian~\cite[Lemma
7.27]{CheegerTian}, characterizes the homogeneous condition in the
case of harmonic $1$-forms.

\begin{lemma}
  \label{homogeneous1}
  Let $u$ be harmonic $1$-form on a limit Ricci-flat cone $C(Y)$ of
  dimension $m$ such that locally
  $\sup |u| + |\nabla_{r\partial_r} u| < \infty$. Suppose
  \begin{align*}
    \nabla_{r\partial_r}u =h(r)u
  \end{align*}
  for some function $h(r)$. Then up to linear combination, $u$ can be
  written as one of the following:
  \begin{enumerate}[(I)]
  \item $u = rdr$ or $r^{-(m-1)}dr$.
  \item $u = d(r^{s+1}g(x))$, where $g(x)$ is an eigenfunction on $Y$
    with eigenvalue $(s+1)(s+m-1)$. $r^{s+1}g(x)$ is a harmonic function
    on $C(Y)$.
  \item $u = r^sg(x)dr - \frac{r^{s+1}}{s+m-3}dg(x)$, where $g$ is an
    eigenfunction on $Y$ with eigenvalue $(s-1)(s+m-3)$.
  \item $u = r^{s+1}\eta(x)$, where $\eta(x)$ is a coclosed eigen
    $1$-form on $Y$ with eigenvalue $(s+1)(s+m-3)$.
  \item $u = (\log{r})d(r^{-(m-4)/2}g(x))$, where $g(x)$ is an
    eigenfunction on $Y$ with eigenvalue $(-m^2/4)+m$.
  \item $u = r^{-(m-2)/2}(\log{r})\eta(x)$, where $\eta(x)$ is a
    coclosed eigen $1$-form on $Y$ with eigenvalue $(-m^2/4)+2m-4$.
  \end{enumerate}
  Note that type (V),(VI) are ruled out by the bounds on $u$. They are
  listed here for completeness. If $u$ is one of the above types, the
  power of $r$ in $|u|$ is called the growth rate of $u$. In (1), the
  growth rate is $1$ or $-(m-1)$. In (II), (III) and (IV), the growth
  rate is $s$.
\end{lemma}

\begin{proof}
  Let $C=C(Y)$ be a limit Ricci-flat cone of real dimension $m$. Any
  $1$-form $u$ on $C$ can be written as
  \begin{equation}
    \label{decomp}
    u=\kappa(r,x)dr+ \eta(r,x),
  \end{equation}
  where $\eta$ is the part tangent to the cross section $Y$. Following
  the calculation in Appendix B of Hein-Sun \cite{HeinSun}, denote prime
  as the derivative with respect to $r$ and denote anything with tilde
  the operators on $Y$. The operators and tensors on $Y$ extend
  trivially to $C(Y)$ by scaling. We have
  \begin{align*}
    \Delta u =
    &\left(\frac{1}{r^2}\tilde{\Delta}\kappa - \kappa'' - \frac{m-1}{r}\kappa'+ \frac{m-1}{r^2}\kappa - \frac{2}{r^3}\tilde{d}^*\eta\right)dr \\
    &+ \frac{1}{r^2}\tilde{\Delta}\eta - \eta'' -
    \frac{m-3}{r}\eta'-\frac{2}{r}\tilde{d}\kappa.
  \end{align*}
  Suppose $\Delta u = 0$. Then this is equivalent to
  \begin{equation}
    \label{eq:1}
    \frac{1}{r^2}\tilde{\Delta}\kappa - \kappa'' - \frac{m-1}{r}\kappa'+
    \frac{m-1}{r^2}\kappa - \frac{2}{r^3}\tilde{d}^*\eta = 0
  \end{equation}
  and
  \begin{equation}
    \label{eq:2}
    \frac{1}{r^2}\tilde{\Delta}\eta - \eta'' -
    \frac{m-3}{r}\eta'-\frac{2}{r}\tilde{d}\kappa = 0.
  \end{equation}
  Suppose $u$ satisfies the following condition:
  \begin{equation}
    \label{eq:3}
    \nabla_{r\partial_r} u = h(r) u.
  \end{equation}
  Using the decomposition \eqref{decomp} and
  \begin{equation*}
    \nabla_{r\partial_r} dx^i = -dx^i,
  \end{equation*}
  \eqref{eq:3} becomes
  \begin{align*}
    r\kappa' &= h\kappa, \\
    r\eta_i' &= (h+1)\eta_i,
  \end{align*}
  where $\eta = \eta_i(r,x)dx^i$. We can solve these ODEs in $r$ and get
  \begin{align*}
    \kappa &= f(r)g(x), \\
    \eta_i &= f(r)rh_i(x)
  \end{align*}
  for some functions $g(x), h_i(x)$ on $Y$, where
  \begin{align*}
    f(r)  = e^{\int \frac{h(r)}{r}dr}.
  \end{align*}
  In sum, $u$ can be written as
  \begin{equation}
    \label{eq:5}
    u = f(r)g(x)dr + f(r)r\eta_1(x),
  \end{equation}
  where $\eta_1(x)$ is a $1$-form on $Y$. Now we plug in \eqref{eq:5}
  into \eqref{eq:1} and \eqref{eq:2}. After rearranging, \eqref{eq:1}
  becomes
  \begin{equation}
    \label{eq:6}
    \frac{\tilde{\Delta}g - 2\tilde{d}^*\eta_1}{g} =
    \frac{r^2f''+(m-1)rf'-(m-1)}{f}=c_1,
  \end{equation}
  where $c_1$ is a constant. We thus have an ODE of $f$:
  \begin{equation}
    \label{eq:7}
    r^2f''+(m-1)rf'-(c_1+m-1)f=0.
  \end{equation}
  The equation of indicial roots, i.e. plugging $f=r^s$ to the ODE, is
  \begin{equation}
    \label{eq:8}
    s^2+(m-2)s-(c_1+m-1)=0.
  \end{equation}
  Without further assumptions on the geometry, the indicial roots could
  be distinct or repeated. In the former case, $f$ can be written as a
  linear combination of
  \begin{equation*}
    r^{s_\pm},
  \end{equation*}
  where
  \begin{equation*}
    s_\pm = \frac{-(m-2)\pm\sqrt{m^2+4c_1}}{2}.
  \end{equation*}
  In the latter case, $f$ can be written as a linear combination of
  \begin{equation*}
    r^{-\frac{m-2}{2}}, r^{-\frac{m-2}{2}}\log{r}.
  \end{equation*}

  Plugging the solution $f$ to the ODE \eqref{eq:7} to \eqref{eq:2} and
  then canceling the $f$'s and $r$'s, we get
  \begin{equation*}
    \tilde{\Delta}\eta_1-(c_1+2m-4)\eta_1-2\tilde{d}g = 0.
  \end{equation*}
  We now have a system of equations on $Y$:
  \begin{align}
    \tilde{\Delta}g-2\tilde{d}^*\eta_1 &= c_1g, \label{eq:9}\\
    \tilde{\Delta}\eta_1 - (c_1+2m-4)\eta_1 &= 2\tilde{d}g. \label{eq:10}
  \end{align}

  Note that at this point, we see that we can decompose the harmonic
  $1$-form $u$ according to the decomposition of $f$ into powers of
  $r$. So we may assume that $f=r^s$. The case when
  $f=r^{-(m-2)/2}\log{r}$ can be worked out similarly.

  Taking $\tilde{d}^*$ of \eqref{eq:10} and set $g_1=\tilde{d}^*\eta_1$,
  the system becomes
  \begin{align}
    \tilde{\Delta}g-2g_1 &= c_1g, \label{eq:11}\\
    \tilde{\Delta}g_1 - (c_1+2m-4)g_1 &= 2\tilde{\Delta}g. \label{eq:12}
  \end{align}
  Substituting the $g_1$'s in \eqref{eq:12} with \eqref{eq:11} and
  completing the square, we get
  \begin{equation}
    \label{eq:13}
    (\tilde{\Delta}-(c_1+m))^2g = (m^2+4c_1)g.
  \end{equation}
  Let
  \begin{equation*}
    \label{eq:14}
    g = \sum_\lambda g_\lambda
  \end{equation*}
  be the spectral decomposition of $g$ with respect to the Laplacian
  $\tilde{\Delta}$. See for example \cite{ding2002heat} and the
  references therein. Then \eqref{eq:13} yields the relation
  \begin{align*}
    \label{eq:15}
    \lambda &= (c_1+m) \pm \sqrt{m^2+4c_1} \\
    &= (s+m-1)(s+1) \text{ or } (s-1)(s+m-3) \\
    &= \lambda_{\pm}.
  \end{align*}
  Thus
  \begin{equation*}
    g = g_+ + g_-,
  \end{equation*}
  where
  \begin{equation*}
    \tilde{\Delta} g_\pm = \lambda_\pm g_\pm.
  \end{equation*}

  For now, we assume that $\lambda_\pm \ne 0$. Any $\eta_1$ that satisfies
  \begin{equation}
    \label{eq:17}
    \tilde{d}^*\eta_1 = \frac{\lambda_+-c_1}{2}g_+ + \frac{\lambda_--c_1}{2}g_-
  \end{equation}
  solves our system of equations. A particular solution is
  \begin{align*}
    \label{eq:18}
    \frac{\lambda_+-c_1}{2\lambda_+}\tilde{d}g_+ + \frac{\lambda_--c_1}{2\lambda_-}\tilde{d}g_-
    = \frac{1}{s+1}\tilde{d}g_+ -\frac{1}{s+m-3}\tilde{d}g_-.
  \end{align*}
  Setting $\eta_2 = \eta_1 - \frac{1}{s+1}\tilde{d}g_+ +\frac{1}{s+m-3}\tilde{d}g_-$, we get
  \begin{equation}
    \label{eq:20}
    u = \frac{1}{s+1}d(r^{s+1}g_+) + r^{s+1}\eta_2 +r^s g_-dr -\frac{r^{s+1}}{s+m-3}\tilde{d}g_-.
  \end{equation}
  It follows that $r^{s+1}\eta_2$ is harmonic and $\eta_2$ is a
  $\tilde{d}^*$-closed eigen $1$-form:
  \begin{equation}
    \label{eq:21}
    \tilde{\Delta}\eta_2 = (s+1)(s+m-3)\eta_2.
  \end{equation}

  The case when one of $\lambda_\pm$ is $0$ can be reduced to the
  special case when $\lambda_\pm =\lambda = 0$ using the calculation
  above.  When $\lambda = 0$, $g$ is a constant function. The case when
  $g=0$ is already covered above. May assume $g=1$. By
  \eqref{eq:11},\eqref{eq:12}, $c_1 = 0$ or $-(2m-4)$.  $c_1 = 0$
  implies $s = 1$ or $-(m-1)$. The case $c_1=-(2m-4)$ is not possible
  when $m \ne 2$, since a locally $L^\infty$, locally $W^{1,2}$ harmonic
  $1$-form on a closed manifold with singularities in codimension $4$ is
  automatically coclosed (and also closed). This can be seen easily
  using the cutoff functions in Lemma \ref{goodcutoff} and an
  integration by parts.
\end{proof}

\begin{remark}
  (1) The lemma actually holds for any smooth Riemannian cones without
  the bounds on $u$. (2) The main difference between the proof in
  \cite{HeinSun} and our proof is that our proof avoids the use of
  spectral decomposition for coclosed $1$-forms. This allows our proof
  to work in our singular setting.
\end{remark}

A special case of Lichnerowicz theorem also holds in our singular
setting.

\begin{lemma}[{\cite[Lemma~B.2]{HeinSun}}]
  Suppose $C(Y)$ is a limit Ricci-flat cone of real dimension $m \ge 3$,
  so that we have ${Ric}_Y = (m-2)g_Y$. Let $\eta$ be a locally
  $L^\infty$, locally $W^{1,2}$ coclosed $1$-form on $\cR \cap Y$. If
  $\Delta \eta = \lambda \eta$ for some $\lambda \in \R$, then
  $ \lambda \ge 2m-4$. When $\lambda = 2m-4$, $\eta$ is dual to a Killing
  vector field. Alternatively, suppose $\eta$ is a closed $1$-form with
  $\Delta \eta = \lambda \eta$. Then $\lambda \ge m-1$.
\end{lemma}

\begin{proof}
  We prove the coclosed case. The closed case is entirely similar. Let
  $\phi_\epsilon$ be the cutoff function supported in the
  $\epsilon$-neighborhood of the singular set of $Y$ as given in Lemma
  \ref{goodcutoff}. Let $\eta$ be a coclosed eigen $1$-form on (the
  regular set of) $Y$ with eigenvalue $\lambda$. By the Bochner formula,
  we compute
  \begin{align*}
    \int_Y \lambda (1-\phi_\epsilon) |\eta|^2
    &=\int_Y \langle \Delta\eta, (1-\phi_\epsilon) \eta \rangle \\
    &=\int_Y (1-\phi_\epsilon) |\nabla\eta|^2 + \int_Y (m-2)(1-\phi_\epsilon) |\eta|^2
    - \int_Y \langle \nabla\eta, \nabla\phi_\epsilon\otimes \eta \rangle.
  \end{align*}
  On the other hand, since $\eta$ is coclosed,
  \begin{align*}
    \int_Y \lambda (1-\phi_\epsilon) |\eta|^2
    &=\int_Y \langle d^*d\eta, (1-\phi_\epsilon) \eta \rangle \\
    &=\int_Y (1-\phi_\epsilon) |d\eta|^2 - \int_Y \langle d\eta,d\phi_\epsilon \wedge \eta \rangle.
  \end{align*}
  As in \cite{HeinSun}, we can decompose $\nabla \eta$ into trace-less
  symmetric, trace, and skew-symmetric parts as
  \begin{align*}
    |\nabla \eta|^2 = |\nabla^{sym}_0 \eta|^2 +
    \left|\frac{\operatorname{tr}(\nabla\eta)}{m-1}\right|^2 +
    |\nabla^{skew}\eta|^2.
  \end{align*}
  Note that $|\nabla^{skew}\eta|^2 = \frac{1}{2}|d\eta|^2$. Thus
  \begin{align*}
    \int_Y (1-\phi_\epsilon)|\nabla\eta|^2 \ge \frac{1}{2} \int_Y (1-\phi_\epsilon) |d\eta|^2,
  \end{align*}
  and we have
  \begin{align*}
    \int_Y \lambda(1-\phi_\epsilon)|\eta|^2 \ge \int_Y 2(m-2)(1-\phi_\epsilon)|\eta|^2
    + 2\int_Y \langle d\eta, d\phi_\epsilon \wedge \eta \rangle
    - 2 \int_Y \langle \nabla\eta, \nabla\phi_\epsilon\otimes\eta \rangle.
  \end{align*}
  Thus it's enough to show that the last two terms tend to $0$ as
  $\epsilon \to 0$. But this follows from the Cauchy-Schwarz inequality,
  $u$ is locally $L^\infty$, $u$ is locally $W^{1,2}$
  (Proposition~\ref{w12}) and
  \begin{align*}
    \lim_{\epsilon \to 0} \int_Y |\nabla\phi_\epsilon|^2 = 0.
  \end{align*}
  Finally, the equality holds precisely when $\nabla\eta$ is
  anti-symmetric; i.e. the dual of $\eta$ is Killing.
\end{proof}

We can now rule out the unwanted parts in the decomposition in Lemma
\ref{homogeneous1}.

\begin{cor}
  \label{exact}
  Let $C(Y)$ be a limit Ricci-flat cone of real dimension at least $4$.
  Let $u$ be a locally $L^\infty$ harmonic $1$-form on $C(Y)$ satisfying
  \begin{equation*}
    \nabla_{r\partial_r} u = h(r) u
  \end{equation*}
  for some smooth function $h$. Suppose further that the growth rate of
  $u$ is less than $1$, i.e.
  \begin{equation*}
    |u| \le C(1+r)^{\bar{s}}
  \end{equation*}
  for some $\bar{s} < 1$. Then $u$ is actually exact:
  \begin{equation*}
    u = d(r^{s+1}g(x))
  \end{equation*}
  for some $0\le s \le \bar{s}$, where $g$ is an eigenfunction on $Y$
  with eigenvalue $(s+1)(s+m-1)$. Note that $r^{s+1}g(x)$ is a harmonic
  function on $C(Y)$.

  Assume furthermore that $C(Y)$ is K\"ahler, i.e. $C(Y)$ is
  Calabi-Yau. Then any locally $L^2$ harmonic function $f$ on $C(Y)$
  with
  \begin{align*}
    |f| \le C(1+r)^{s'}, \:\: 0 \le s' < 2
  \end{align*}
  is the real part of a holomorphic function. In particular, $f$ is
  pluriharmonic.
\end{cor}

\begin{proof}
  Assume $C(Y)$ is a limit Ricci-flat cone of real dimension at least
  $4$. To prove the first part, we rule out types (III)-(VI). Types
  (V) and (VI) are ruled out by the assumption on $u$. For Type (III),
  Lichnerowicz implies that $s\ge 2$ or $s\le 2-m$. The first case
  violates the growth assumption, while the second case violates the
  $L^\infty_{loc}$ assumption. For type (IV), we have $s\ge 2$ or
  $s\le 2-m$. Again, these are ruled out by our assumption.

  Now, assume $C(Y)$ is Calabi-Yau. Let $f$ be a locally $L^2$
  harmonic function on $C(Y)$. By the spectral decomposition, we may
  assume $f$ is homogeneous. Thus $d^c f$ is a homogeneous harmonic
  $1$-form and has growth rate less than $1$. From the above we see
  that $d^c f = dh$ for some harmonic function $h$. So $f$ is
  pluriharmonic.
\end{proof}

\section{A local $L^2$ estimate for the exterior derivative}

In this section, we prove Theorem \ref{local_l2}. One key ingredient
is Lemma \ref{epsilon_monotonicity}, which roughly states that on a
Ricci-flat metric ball which is close to a limit Ricci-flat cone, a
harmonic $1$-form that is orthogonal to exact $1$-forms must grow at
least linearly.  This is true when the ball actually lies in a
Ricci-flat cone and the $1$-form is homogeneous, as we have seen in
Corollary \ref{exact}. To obtain the local $L^2$ estimate for the
exterior derivative, we use the fact that all but finitely many scales
in a Ricci-flat metric ball $B(p,1)$ are close to a Ricci-flat cone in
the Gromov-Hausdorff sense. This is a direct consequence of the
Cheeger-Colding cone rigidity theorem \cite{CheegerColding}. At each
scale $2^{-k}$ such that the manifold is actually close to a cone, we
apply Lemma \ref{epsilon_monotonicity} to obtain a growth estimate for
the orthogonal projection with respect to the $L^2$ inner product on
the $(2^{-k})$-ball. Together with an integration by parts, we then
concatenate these growth estimates above to conclude the proof of
Theorem \ref{local_l2}.

We need the following Green's formula on limit Ricci-flat cones:

\begin{lemma}
  Let $C(Y)$ be a limit Ricci-flat cone of dimension $m \ge 2$. Suppose
  $u, v$ are harmonic $1$-forms on $B_r = B(o,r) \subset C(Y)$ such that
  locally
  $\sup (|u| +|\nabla_{r\partial_r} u| + |v| + |\nabla_{r\partial_r} v|)
  < \infty$. Then
  \begin{align*}
    \int_{\partial B_r} \langle \nabla_{\partial_r} u, v \rangle =
    \int_{\partial B_r} \langle u, \nabla_{\partial_r} v \rangle.
  \end{align*}
\end{lemma}

\begin{proof}
  If the cone $C(Y)$ is smooth, then this follows from Green's formula
  and the fact that both $u,v$ are harmonic. Let $\phi_\epsilon$ be the
  good cutoff function on $B(o,r)$ supported outside the
  $\epsilon$-neighborhood of the singular set of $B(o,r)$. Define
  $u_\epsilon = \phi_\epsilon u$. Then by Green's formula,
  \begin{align*}
    \int_{\partial B_r} \langle \nabla_{\partial_r} u_\epsilon, v \rangle -
    \int_{\partial B_r} \langle u_\epsilon, \nabla_{\partial_r} v \rangle =
    \int_{B_r} \langle \nabla \cdot \nabla u_\epsilon, v \rangle -
    \int_{B_r} \langle u_\epsilon, \nabla \cdot \nabla v \rangle.
  \end{align*}
  By the Bochner formula and the Ricci-flat condition, the second term
  of the right hand side is $0$, and the first term on the right hand
  side can be computed as
  \begin{align*}
    \int_{B_r} \langle \nabla \cdot \nabla u_\epsilon, v \rangle =
    \int_{B_r} \langle (-\Delta \phi_\epsilon) u + 2  \nabla \phi_\epsilon \cdot \nabla u, v \rangle.
  \end{align*}
  Note that
  \begin{align*}
    \left|\int_{B_r} \langle (-\Delta \phi_\epsilon) u, v \rangle\right| \le
    C \int_{B_r} |\Delta \phi_\epsilon| \to 0
  \end{align*}
  as $\epsilon \to 0$, and that
  \begin{align*}
    \left|\int_{B_r} \langle \nabla \phi_\epsilon \cdot \nabla u, v \rangle\right|
    \le C \sqrt{\int_{B_r} |\nabla \phi_\epsilon|^2} \to 0
  \end{align*}
  as $\epsilon \to 0$ by the Cauchy-Schwarz inequality.
\end{proof}

The Green's formula implies the following orthogonality property of
homogeneous harmonic $1$-forms.

\begin{prop}
  \label{homogeneous_prop}
  Let $C(Y)$ be a limit Ricci-flat cone of dimension $m$. Suppose $u, v$
  are harmonic $1$-forms on $B_1 = B(o,1) \subset C(Y)$ such that both
  locally
  $\sup (|u| + |\nabla_{r\partial_r}u|+ |v| + |\nabla_{r\partial_r}v|) <
  \infty$. Suppose furthermore that $v$ is homogeneous of degree $s$,
  i.e.
  \begin{align*}
    \nabla_{r\partial_r} v = s v.
  \end{align*}
  Then there exists a constant $C \in \R$ such that
  \begin{align*}
    \int_{\partial B_r} \langle u, v \rangle = Cr^{2s+m-1}
  \end{align*}
  for all $r\in [0,1]$.
\end{prop}

\begin{proof}
  Write
  \begin{align*}
    I(r) = \dashint_{\partial B_r} \langle u, v \rangle.
  \end{align*}
  Then by the previous lemma and a cutoff function argument, we have
  \begin{align*}
    r\partial_r I(r) &=
    \dashint_{\partial B_r} \langle \nabla_{r\partial_r} u, v \rangle +
    \dashint_{\partial B_r} \langle u, \nabla_{r\partial_r} v \rangle \\
    &= 2 \dashint_{\partial B_r} \langle u, \nabla_{r\partial_r} v \rangle \\
    &= 2s I(r).
  \end{align*}
  Integrating, we get
  \begin{align*}
    I(r) = Cr^{2s}
  \end{align*}
  for some constant $C \in \R$. This completes the proof.
\end{proof}

We are ready to state our key lemma:

\begin{lemma}
  \label{epsilon_monotonicity}
  For any $\delta > 0$ and $v > 0$, there exists $\epsilon > 0$ with the
  following significance: let $B(p,2)$ be a Riemannian metric ball such
  that $\vol(B(p,1)) > v$ and $|\ric(B(p,2))| < \epsilon$, and let
  $B(o,2) \subset C(Y)$ be a limit Ricci flat cone such that
  \begin{align*}
    \dgh(B(p,2), B(o,2)) < \epsilon.
  \end{align*}
  Let $u$ be a harmonic $1$-form on $B(p,1)$ such that $u$ is
  $L^2$-orthogonal to the space of closed and coclosed harmonic
  $1$-forms on $B(p,1)$. Then $u$ grows ``almost linearly'' in the
  following sense:
  \begin{align*}
    \dashint_{B(p,1)} |u|^2 \ge 2^{2(1-\delta)} \dashint_{B(p,1/2)}
    |u|^2.
  \end{align*}
\end{lemma}

\begin{proof}
  We argue by contradiction. Let $B(p_i, 2)$ be a sequence of
  Riemannian metric balls with $\vol(B(p_i,1)) > v$ and
  $|\ric(B(p_i,2))| < \epsilon_i$, and let $C(Y_i)$ be a sequence of
  limit Ricci-flat cones with vertex $o_i$. Suppose that
  \begin{align*}
    \dgh(B(p_i,2), B(o_i,2)) < \epsilon_i
  \end{align*}
  for each $i$, where $\epsilon_i \to 0$.

  Suppose for contradiction that there exist $\delta > 0$ and a sequence
  $u_i$ of harmonic $1$-forms on $B(p_i,1)$ satisfying the
  $L^2$-orthogonal condition on $B(p_i,1)$, such that
  \begin{align*}
    \dashint_{B(p_i,1)} |u_i|^2 < 2^{2(1-\delta)} \dashint_{B(p_i,1/2)}
    |u_i|^2.
  \end{align*}

  After passing to a subsequence, we may assume that both $B(p_i,2)$
  and $B(o_i, 2)$ converge to $B(o,2)$ in a limit Ricci-flat cone
  $C(Y)$. We may normalize $u_i$ so that
  \begin{align*}
    \dashint_{B(p_i,1/2)} |u_i|^2 = 1.
  \end{align*}
  So
  \begin{align*}
    \dashint_{B(p_i,1)} |u_i|^2 < 2^{2(1-\delta)}
  \end{align*}
  is uniformly bounded. Thus by passing to a subsequence, we may
  assume that $u_i$ converges to a nonzero harmonic $1$-form $u$ on
  $B(o,1)$. The convergence is smooth on any compact subset of the
  regular set of $B(o,1)$, and both $u_i$ and $u$ are uniformly
  bounded on $B(o,1/2)$ by Proposition~\ref{linfinity}. Taking limit
  of the above inequality, we see that
  \begin{align*}
    \dashint_{B(o,1)} |u|^2 \le 2^{2(1-\delta)}.
  \end{align*}
  Note that since $u$ is uniformly bounded and the convergence is
  smooth away from the singular set, we have
  \begin{align*}
    \dashint_{B(o,1/2)} |u|^2 = \lim_{i\to\infty} \dashint_{B(p_i,1/2)} |u_i|^2 = 1
  \end{align*}
  This follows from an argument similar to
  \cite[Lemma~2.16]{DonaldsonSunII} using Colding's volume convergence
  theorem~\cite{ColdingVolume}.
  
  By Lemma~\ref{nablar}, we have that locally
  $|\nabla_{r\partial_r}u| < \infty$. Applying the $L^2$ three circle
  theorem \ref{threecircle}, we have
  \begin{align}\label{eq:growth}
    \dashint_{B(o,2^{-i})} |u|^2 \le 2^{2(1-\delta)} \dashint_{B(o,2^{-i-1})} |u|^2.
  \end{align}
  We can extract the ``lowest order term'' of $u$ as follows. Let
  \begin{align*}
    v_i = \frac{\phi_i^*u}{\|\phi_i^*u\|_{L^2(B(o,1))}},
  \end{align*}
  where $\phi_i: C(Y) \to C(Y)$ is the scaling by $2^{-i}$. Thus $v_i$
  is uniformly bounded on any open subsets of the regular set of
  $B(o,1/2)$. After passing to a subsequence, $v_i$ converges to a
  nonzero harmonic $1$-form $v$ on $B(o,1/2)$. By the growth estimate
  \eqref{eq:growth} we see that the $L^2$ ratio of $v$ is constant. So
  $v$ is homogeneous by Theorem \ref{threecircle}.  By Corollary
  \ref{exact}, $v=df$, where $f$ is a harmonic function with growth
  rate $0\le s < 2$. Since
  \begin{align*}
    \int_{B(o,1/2)} \langle v, v \rangle > 0,
  \end{align*}
  we have, for large enough $i$,
  \begin{align*}
    0 < \int_{B(o,1/2)} \langle \phi_i^*u, v \rangle
    = 2^{mi} \int_{B(o,2^{-i-1})} \langle u, (\phi_i^{-1})^* v \rangle
    = 2^{mi+si} \int_{B(o,2^{-i-1})} \langle u, v \rangle,
  \end{align*}
  where $m$ denotes the dimension of $C(Y)$, $s$ denotes the degree of
  $v$, and we use the homogeneity of $v$ in the last equality. By
  Proposition \ref{homogeneous_prop}, it follows that
  \begin{align*}
    \int_{B(o,1)} \langle u, v \rangle > 0.
  \end{align*}

  On the other hand, by \cite[Theorem~2.1]{ding2004existence} (see
  also \cite[Lemma~4.1]{Xu}), after passing to a subsequence there
  exists a sequence of harmonic functions $f_i$ on $B(p_i,2)$ such
  that $f_i$ converges to $f$ uniformly in the Gromov-Hausdorff
  sense. By the Schauder estimates, it follows that $f_i$ converges to
  $f$ in $C^\infty$ on compact subsets of $B(o,1) \cap \cR$ in the
  Gromov-Hausdorff sense. Also by the Cheng-Yau gradient estimate,
  $df_i$ and $df$ are uniformly bounded on $B(p_i,1)$ and $B(o,1)$,
  respectively. It follows from an argument similar to
  \cite[Lemma~2.16]{DonaldsonSunII} that we can take limit of the
  $L^2$ orthogonal condition and we get
  \begin{align*}
    0 = \lim_{i\to\infty} \int_{B(p_i,1)} \langle u_i, df_i \rangle =
    \int_{B(o,1)} \langle u, df \rangle,
  \end{align*}
  which is a contradiction.
\end{proof}

We can now prove our main theorem for harmonic $1$-forms.

\begin{proof}[Proof of Theorem \ref{local_l2}]
  Let $\epsilon > 0$ be given as in Lemma \ref{epsilon_monotonicity}. We
  first prove the case when $|\mathrm{Ric}| \le \epsilon$.

  From Cheeger-Colding's cone rigidity theorem~\cite{CheegerColding}
  and volume comparison (see
  e.g. \cite[p.334]{CheegerNaberInventiones} or
  \cite[Proposition~2.29]{ChiuThesis} for details), there exists a
  number $N(v,\epsilon)$ such that for all but $N(v,\epsilon)$ of
  $k \in \N$, we have
  \begin{align*}
    \dgh(B(p,2^{-k+1}), B(o_k, 2^{-k+1})) < \epsilon 2^{-k+1}
  \end{align*}
  for some $B(o_k, 2^{-k+1})$ inside a metric cone $C(Y_k)$. By a
  compactness argument, we can assume these metric cones $C(Y_k)$ are
  limit Ricci-flat cones. Set $u_0 = u$. We define $u_k$ on
  $B(p, 2^{-k})$ inductively.

  For each of the good $k$, let $u_k$ be the $L^2$ orthogonal projection
  of $u_{k-1}$ onto the $L^2$ complement of the space of closed,
  coclosed $1$-forms on $B(o, 2^{-k})$. Note that this space is closed in
  $L^2$: let $w_i$ be a sequence of closed, coclosed $L^2$ $1$-forms on
  $B(p,1)$ such that $w_i \to w$ in $L^2$ for some $L^2$ 1-form on
  $B(p,1)$. Then this implies that $dw = 0$ and $d^*w = 0$ in the weak
  sense. Elliptic regularity then implies that $dw =0$ and $d^*w = 0$ in
  the strong sense.

  By Lemma \ref{epsilon_monotonicity},
  \begin{align*}
    \dashint_{B(p,2^{-k-1})} |u_{k+1}|^2 \le 2^{-2(1-\delta)}\dashint_{B(p,2^{-k})} |u_k|^2.
  \end{align*}

  For finitely many bad $k$, set $u_k = u_{k-1}$. The following
  inequality holds by volume comparison:
  \begin{align*}
    \dashint_{B(p, 2^{-k-1})} |u_{k+1}|^2 \le 2^{-2(1-\delta)}C \dashint_{B(p,2^{-k})} |u_k|^2
  \end{align*}
  for some constant $C$ depending on $\delta$ and $v$. Concatenating
  the above two types of inequalities, we get
  \begin{align*}
    \dashint_{B(p,2^{-k})} |u_k|^2 \le
    C(\epsilon,\delta,v)2^{-2k(1-\delta)} \dashint_{B(p,1)} |u|^2.
  \end{align*}

  Now, pick $2^{-k-1} \le 2r \le 2^{-k}$, and we have
  \begin{align*}
    \dashint_{B(p,2r)} |u_k|^2 \le C\dashint_{B(p,2^{-k})} |u_k|^2.
  \end{align*}
  Since $\Delta |u_k|^2 = -2|\nabla u_k|^2-\ric(u_k, u_k)$, an integration by
  parts with a good cutoff function gives
  \begin{align*}
    \dashint_{B(p,r)} |\nabla u_k|^2
    \le \left(\frac{\epsilon}{2} + Cr^{-2}\right) \dashint_{B(p,2r)} |u_k|^2.
  \end{align*}
  As $r \le 1/2$, we may absorb $\epsilon/2$ into $Cr^{-2}$ on the right
  hand side of the above inequality. We also have the pointwise
  inequality
  \begin{align*}
    |du|^2 + |d^*u|^2 = |du_k|^2 + |d^*u_k|^2 \le 2|\nabla u_k|^2.
  \end{align*}
  The desired estimate then follows by combining the above four
  inequalities.

  Now we prove the case when $|\mathrm{Ric}| \le 1$. Let
  $r \in (0,1/2]$. First let us assume $r \le \sqrt{\epsilon}/2$,
  i.e. $r$ is small. By rescaling the metric $g = \epsilon \tilde{g}$,
  we can apply what we just proved to get
  \begin{align*}
    \dashint_{B(p,r)} |du|^2 + |d^*u|^2 \le
    C'\epsilon^{\delta-1}r^{-2\delta} \dashint_{B(p,\sqrt{\epsilon})}
    |u|^2 \le C''(\delta, v, \epsilon) r^{-2\delta}
    \dashint_{B(p,1)} |u|^2,
  \end{align*}
  where the last inequality follows from volume comparison.
  
  Now let us assume $r > \sqrt{\epsilon}/2$. In this case, we get a
  better estimate. Integrating by parts with a good cutoff function
  gives

  \begin{align*}
    \int_{B(p,r)} |du|^2 + |d^*u|^2\le 2\int_{B(p,r)} |\nabla u|^2 \le
    \frac{C_1}{r^2} \int_{B(p,2r)} |u|^2.
  \end{align*}

  So
  \begin{align*}
    \dashint_{B(p,r)} |du|^2 + |d^*u|^2 \le
    \frac{C_1}{r^2} \frac{\vol(B(p,1))}{\vol(B(p,r))}\dashint_{B(p,1)}|u|^2 \le
    C_2(\delta,v,\epsilon) \dashint_{B(p,1)} |u|^2
  \end{align*}
  by volume comparison.
\end{proof}

\begin{remark}
  \label{dstar}
  We have a pointwise bound for $d^*u$. Since $|d^*u|$ is subharmonic,
  we can apply the Li-Schoen mean value inequality to show that
  $|d^*u|^2$ is bounded by the $L^2$ average of $|d^*u|$, which in
  turn is bounded by the $L^2$ average of $u$ by an integration by
  parts with a good cutoff functions.
\end{remark}

\begin{proof}[Proof of Theorem \ref{ddbar}]
  Let $f$ be a harmonic function on $B(p,1)$. Applying Theorem
  \ref{local_l2} to the harmonic $1$-form $u = d^c f$ gives
  \begin{align*}
    \dashint_{B(p,r)} |\partial\bar{\partial} f|^2 &\le C r^{-2\delta}
    \dashint_{B(p,3/4)} |d^c f|^2 \\
    &= C r^{-2\delta} \dashint_{B(p,3/4)} |df|^2 \\
    &\le C' r^{-2\delta} \dashint_{B(p,1)} f^2,
  \end{align*}
  where the last inequality follows from an integration by part with a
  good cutoff function and a volume comparison.
\end{proof}

Now that we have the analogue of the gradient estimate, we are ready
to prove the Liouville type theorem for harmonic $1$-forms.

\begin{proof}[Proof of Theorem \ref{maintheorem}]
  It is enough to prove that for every $p \in M$,
  \begin{align*}
    \dashint_{B(p,1)} |u|^2 = 0.
  \end{align*}
  Let $r < 1/2$. We rescale the metric by $r^2$, use Theorem
  \ref{maintheorem}, then rescale back by $1/r^2$ and get
  \begin{align*}
    \dashint_{B(p,1)} |du|^2 + |d^*u|^2 \le
    C r^{2-2\delta} \dashint_{B(p,1/r)} |u|^2 \le
    Cr^{2-2\delta-2s}.
  \end{align*}
  We choose $\delta > 0$ such that $2-2\delta-2s > 0$. Note that since
  $M$ has maximal volume growth, the constant $C$ only depends on a
  fixed volume lower bound and $\delta$. Hence it does not depend on
  $r$. Letting $r \to 0$, we get that $u$ is both closed and coclosed.

  On the other hand, by Corollary 1.5 (3) of Anderson \cite{Anderson90},
  we know that $H^1(M) = 0$. So $u = df$ for some function $f$. Since
  $d^*u = 0$, $f$ is harmonic.
\end{proof}

\begin{remark}
  \label{endremark}
  As mentioned in the introduction, using the method in Conlon-Hein
  \cite{ConlonHeinI}, we can relax the subquadratic growth condition
  to $o(r^2)$ in the AC case. Actually, we will show that we can
  replace the sublinear condition in Theorem \ref{maintheorem} by
  $o(r)$, assuming the manifold $M$ is AC Calabi-Yau. The key is that
  the metric $g$ converges to the cone metric $g_0$ on the tangent
  cone at infinity $C$ in $O(r^{-\epsilon})$ for some $\epsilon > 0$
  by definition. This implies that the Laplacian $\Delta$ with respect
  to $g$ and the Laplacian $\Delta_0$ with respect to $g_0$ vary in
  $O(r^{-\epsilon})$ in operator norm. Let $u$ be a harmonic $1$-form
  on $M$ such that $u = o(r)$. We show that $u = O(r^s)$ for some
  $s < 1$. Since $M$ and $C$ are diffeomorphic outside compact
  subsets, we may transplant $u$ to a $1$-form $u_0$ on $C$ such that
  $u = u_0$ outside a compact subset. Note that
  $\Delta_0 u_0 = (\Delta_0 - \Delta) u_0 = O(r^{-1-\epsilon})$ for
  $r$ large. Then we solve the equation $\Delta_0 v_0 = \Delta_0 u_0$
  on $C$ by solving ODEs. Thus $v_0 = O(r^{1-\epsilon})$. Since
  $u_0 - v_0 = o(r)$ is harmonic, it follows that $u_0 - v_0 = O(r^s)$
  for some $s<1$ by the Cheeger-Tian lemma (Corollary~\ref{exact}). So
  $u = (u_0 - v_0) + v_0 = O(r^{s'})$ for $r$ large, where
  $s' = \max\{1-\epsilon,s\}$.
\end{remark}

\end{document}